\documentclass[a4paper,12pt,leqno]{article}
\usepackage[latin1]{inputenc}
\usepackage[T1]{fontenc}
\usepackage{lmodern}
\usepackage[english]{babel}
\usepackage[intlimits]{amsmath}
\usepackage{amssymb}
\usepackage{amsthm}
\usepackage{tikz-cd}
\usepackage{bm}
\usepackage{enumerate}
\usepackage[noadjust]{cite}
\usepackage{xifthen}
\usepackage{mathtools}
\usepackage[pdfpagemode=UseNone,pdfstartview=Fit]{hyperref}
\usepackage[datesep=.]{datetime2}
\DTMsetdatestyle{dmyyyy}
\theoremstyle{plain}
\newtheorem{thm}{Theorem}
\newtheorem{prop}[thm]{Proposition}
\newtheorem{lemma}[thm]{Lemma}
\newtheorem{cor}[thm]{Corollary}
\theoremstyle{definition}
\newtheorem{rem}[thm]{Remark}
\newtheorem{defn}[thm]{Definition}

\newtheorem{remark}[thm]{Remark}

\allowdisplaybreaks[1]

\newcommand*{\Zset}{\mathbb{Z}}

\newcommand*{\Rset}{\mathbb{R}}

\newcommand*{\abs}[1]{\left\lvert#1\right\rvert}
\newcommand*{\norm}[1]{\left\lVert#1\right\rVert}

\DeclareMathOperator{\im}{im}

\DeclareMathOperator{\coker}{coker}
\DeclareMathOperator{\id}{id}

\DeclareMathOperator{\Homeo}{Homeo}

\newcommand{\To}{\Rightarrow}

\newcommand*{\Vect}{\mathbf{Vect}}
\newcommand*{\Set}{\mathbf{Set}}
\newcommand*{\Ab}{\mathbf{Ab}}
\newcommand*{\Top}{\mathbf{Top}}
\newcommand*{\mb}[1]{\mathbf{#1}}

\newcommand{\Trans}[1]{\mb{Trans}_{\mb{#1}}}
\newcommand{\Transp}[1]{\mb{Trans}_{#1}}
\newcommand*{\K}{\textnormal{K}}
\newcommand*{\dist}{\textnormal{d}}

\newcommand*{\Dgm}{\mathsf{Dgm}}

\AtBeginDocument{
\let\epsilon\varepsilon
}

\begin{document}

\vspace*{1cm}
{\begin{center}{\Large\bfseries Erosion distance for generalized persistence modules}\\
\vspace{-3mm} {\begin{center}\large Ville Puuska \end{center}} \makeatletter{\renewcommand*{\@makefnmark}{}
\footnotetext{University of Tampere, Faculty of Natural Sciences \today}\makeatother} \makeatletter{\renewcommand*{\@makefnmark}{}
\footnotetext{puuskaville@gmail.com}\makeatother}\end{center}}

\begin{abstract}
The persistence diagram of Cohen-Steiner, Edelsbrunner, and Harer was recently generalized by Patel to the case of constructible persistence modules with values in a symmetric monoidal category with images. Patel also introduced a distance for persistence diagrams, the erosion distance. Motivated by this work, we extend the erosion distance to a distance of rank invariants of generalized persistence modules by using the generalization of the interleaving distance of Bubenik, de Silva, and Scott as a guideline. This extension of the erosion distance also gives, as a special case, a distance for multidimensional persistent homology groups with torsion introduced by Frosini. We show that the erosion distance is stable with respect to the interleaving distance, and that it gives a lower bound for the natural pseudo-distance in the case of sublevel set persistent homology of continuous functions.
\end{abstract}

\tableofcontents

\section*{Introduction}\addcontentsline{toc}{section}{Introduction}

Persistent homology has risen to be a popular and powerful tool for extracting topological features of data sets (see \cite{ghrist08} and \cite{carlsson09}). Persistent homology takes a filtration of a topological space and computes the birth and death times of topological features in the filtration. This allows us to distinguish the features that are only noise and have very short lifespans from the more persistent ones. To compute these features and their lifespans, homology is applied to the filtration, which leads to a functor $\Rset \to \Vect$,\footnote{$\Vect$ is the category of all \emph{finite dimensional} $\mathbb F$-vector spaces over some fixed field $\mathbb F$.} which is often called a persistence module. There are two main visualisations of persistence modules: barcodes, which collect the birth and death times of homology classes in the filtration as intervals (see \cite{zomorodian05b}); and persistence diagrams, which collect the same information as points in $\Rset^2$ (see \cite{cohen-steiner07} and \cite{chazal16}).

Since persistent homology is motivated by problems in data-analysis, we need to have a notion of distance between invariants obtained from different data sets, which must be stable with respect to noise in data. For barcodes and persistence diagrams, the bottleneck distance and the Wasserstein distances are the most commonly used distances. For persistence modules themselves, we have the interleaving distance, which has been generalized to extensions of persistence modules, e.g.~to multidimensional and generalized persistence modules (see \cite{lesnick15} and \cite{bubenik15}). For persistence modules $\Rset \to \Vect$, the interleaving distance is computable, because it is equal to the bottleneck distance, but up to our knowledge, there are currently no efficient algorithms to compute the interleaving distance in the multidimensional setting.

In this paper, we present a stable distance for persistence modules $\mb P \to \mb C$, i.e. functors, which is computed directly from invariants of persistence modules known as rank invariants, where $\mb P$ is a preordered set and $\mb C$ is an Abelian category. This distance is an extension of two previous distances: the erosion distance of \cite{patel16}, and the distance $\dist_T$ of \cite{frosini13}. We call this distance the erosion distance after the former. We show that the erosion distance is stable with respect to the interleaving distance, and that it gives a lower bound for the natural pseudo-distance in the case of sublevel set persistent homology of continuous functions.

The distance $\dist_T$ was introduced by Frosini in \cite{frosini13} as a distance for multidimensional persistent homology groups with torsion, i.e.~persistence modules obtained by applying singular homology with coefficients in an Abelian group to a multiparameter filtration of a space. It was shown that $\dist_T$ gives a lower bound for the natural pseudo-distance when the filtrations are obtained as sublevel set filtrations of continuous functions. This distance can be directly extended for all functors $\Rset^n \to \Ab$.

A recent step forward in the effort to extend the theory of persistent homology came when Patel \cite{patel16} generalized the persistence diagram for so called constructible persistence modules $\Rset \to \mb C$, where $\mb C$ is any essentially small symmetric monoidal category with images. Additionally, a new distance for persistence diagrams, the erosion distance, was introduced.

This paper has two main purposes. Firstly, we wish to extend the erosion distance of \cite{patel16}, independent of persistence diagrams, in order to allow it to be used in the multidimensional setting without requiring constructibility. Secondly, we wish to look at the distance $\dist_T$ of \cite{frosini13} from a more categorical perspective. Essentially, defining either of these distances starts with giving a preorder of the target category $\mb C$, and then extending it to a preorder of maps $\Dgm_{\Rset^n} := \{(a,b) \in \Rset^n \times \Rset^n \mid a < b\} \to \mb C$. Then, every persistence module $F \colon \Rset^n \to \mb C$ induces a map
\[\mathcal{F} \colon \Dgm_{\Rset^n} \to \mb C,\ \mathcal{F}(a,b) = \im F(a < b),\]
which is a straightforward generalization of the rank invariant of \cite{carlsson09b}. For maps $f,g \colon \Dgm_{\Rset^n} \to \mb C$, we get an extended pseudo-metric by taking the infimum of all $\epsilon \in [0,\infty)$ such that
\[f(a-\epsilon,b+\epsilon) \leq g(a,b) \text{ and } g(a-\epsilon,b+\epsilon) \leq f(a,b) \text{ for all } (a,b) \in \Dgm_{\Rset^n},\]
which gives us an extended pseudo-metric for rank invariants of persistence modules $\Rset^n \to \mb C$. To extend this to persistence modules $\mb P \to \mb C$, we use translations of the preordered set $\mb P$ and superlinear families or sublinear projections in fundamentally the same way that they are used in \cite{bubenik15}, where they are used to extend the interleaving distance for generalized persistence modules.

\subsection*{Outline}

In section \ref{sec: Erosion distance for maps}, we define the erosion distance in its most general form, i.e.~for (decreasing) maps
\[\Dgm_{\mb P} = \{(a,b) \in \mb P \times \mb P \mid a < b\} \to \mb G,\]
where $\mb G$ is a preordered class, and $\mb P$ is a preordered set equipped with a sublinear projection or a superlinear family. We also show in subsection \ref{sec: The distance induced by the sup-norm} that the $L^{\infty}$-distance of functions $X \to \Rset$, where $X$ is any set can be interpreted as an erosion distance.

In section \ref{sec: Erosion distance for persistence modules}, we first go over the details of the erosion distance of \cite{patel16}, and then define the erosion distance for rank invariants of persistence modules. We prove that it is an extended pseudo-metric (Corollary \ref{cor: metric for persistence modules}), and that it is stable with respect to the interleaving distance (Theorem \ref{thm: stability of erosion wrt interleaving}).

In section \ref{sec: Minimal preorders and the natural pseudo-distance}, we show that the distance $\dist_T$ of \cite{frosini13} is a special case of the erosion distance, and show that the erosion distance gives a lower bound for the natural pseudo-distance (Theorem \ref{thm: lower bound nat pseudodist}).

In section \ref{sec: Adjunction relation}, we consider the situation where $\mb P$ is equipped with a sublinear projection \emph{and} a superlinear family. We show that if the sublinear projection and the superlinear family satisfy the adjunction relation as defined in \cite{bubenik15}, then the two erosion distances are equal.

\section{Erosion distance for maps}\label{sec: Erosion distance for maps}

\noindent Throughout these notes we let $\mb P$ be a preordered set and $\mb G$ be a preordered class. We denote
\[\Dgm_{\mb P} = \{(a,b) \in \mb P \times \mb P \mid a < b\}\footnote{We use the notation $a < b$ to mean that $a \leq b$ and $a \neq b$.}\]
and we define a preorder for the set $\Dgm_{\mb P}$ by setting
\[(a,b) \leq (a',b') \iff a \geq a' \text{ and } b \leq b',\]
i.e.~the preorder inherited from $\mb P^{\text{op}} \times \mb P$.

Let $\mb P = \Rset$ and take a function $f \colon \Dgm_\Rset \to \mb G$. We can think of the function $f$ as an assignment of elements of $\mb G$ to each point in $\Dgm_\Rset$. Now, let $\epsilon \geq 0$ and consider the function
\[f_\epsilon \colon \Dgm_\Rset \to \mb G,\ f_\epsilon(a,b) = f(a-\epsilon, b+\epsilon).\]
We can think of the assignment of elements given by $f_\epsilon$ as moving the points of $f$ down and right by $\epsilon$, or towards the diagonal $\{(x,y) \mid x = y\}$ by $\sqrt 2 \epsilon$, and killing elements that are moved to or below the diagonal. If $g \colon \Dgm_\Rset \to \mb G$ is another function, we can ask how much we need to move $f$ and $g$ towards the diagonal to get the pair of inequalities
\[f_\epsilon \leq g \text{ and } g_\epsilon \leq f.\]
It's easy to see that by taking the infimum over all $\epsilon$ such that these inequalities hold we get an extended pseudo-metric for functions $\Dgm_\Rset \to \mb G$. This idea can be generalized to arbitrary preordered sets $\mb P$ by using translations and superlinear families or sublinear projections in the same way as in \cite{bubenik15}. Specifically, instead of moving points down and right by $\epsilon$, we move them by a pair of translations of $\mb P$.

\begin{defn}
A \emph{translation}\footnote{Note that our definition of a translation is stricter than the one in \cite{bubenik15} since we require a translation to be an automorphism, instead of just an endofunctor of $\mb P$.} of the set $\mb P$ is a map $\Gamma \colon \mb P \to \mb P$ such that
\begin{enumerate}[--]
\item $\Gamma$ is a bijection,
\item $a \leq b \Rightarrow \Gamma a \leq \Gamma b$ and $\Gamma^{-1} a \leq \Gamma^{-1} b$ for all $a,b \in \mb P$,
\item $a \leq \Gamma a$ for all $a \in \mb P$.
\end{enumerate}
In other words, a translation is an automorphism of $\mb P$ with a natural transformation from the identity functor $I \colon \mb P \to \mb P$. We denote the preordered set of translations of $\mb P$ by $\Trans P$.
\end{defn}

Note that $\Trans P$ is closed under composition and for every $\Gamma \in \Trans P$ we have $\Gamma^{-1}a \leq a$ for all $a \in \mb P$.

\begin{defn}
Let $\Gamma, \K \in \Trans P$. A $(\Gamma,\K)$-\emph{erosion} of a map $f \colon \Dgm_{\mb P} \to \mb G$ is the map
\[\nabla_{\Gamma, \K}f \colon \Dgm_{\mb P} \to \mb G,\ (a,b) \mapsto f(\Gamma^{-1} a, \K b).\]
We also use the shorthand $\nabla_{\Gamma}f = \nabla_{\Gamma,\Gamma}f$.
\end{defn}

\begin{prop}[Triangle inequality]\label{thm: triangle inequality}
Let $f,g,h \colon \Dgm_{\mb P} \to \mb G$ and $\Gamma, \Gamma', \K, \K' \in \Trans P$ such that
\[\nabla_{\Gamma,\K}f \leq g,\ \nabla_{\K,\Gamma}g \leq f \text{ and } \nabla_{\Gamma',\K'}g \leq h,\ \nabla_{\K',\Gamma'}h \leq g.\]
Then
\[\nabla_{\Gamma'\Gamma,\K\K'}f \leq h,\ \nabla_{\K\K',\Gamma'\Gamma}h \leq f.\]
\end{prop}
\begin{proof}
Take $(a,b) \in \Dgm_{\mb P}$. Now
\begin{align*}
f((\Gamma'\Gamma)^{-1} a, \K\K' b) &= f(\Gamma^{-1}\Gamma'^{-1}a, \K\K'b)
\\&\leq g(\Gamma'^{-1}a, \K'b)
\\&\leq h(a,b),
\end{align*}
so $\nabla_{\Gamma'\Gamma,\K\K'}f(a,b) \leq h(a,b)$, and
\begin{align*}
h((\K\K')^{-1}a, \Gamma'\Gamma b) &= h(\K'^{-1}\K^{-1}a, \Gamma'\Gamma b)
\\&\leq g(\K^{-1} a, \Gamma b)
\\&\leq f(a,b),
\end{align*}
so $\nabla_{\K\K',\Gamma'\Gamma}h(a,b) \leq f(a,b)$.
\end{proof}

\begin{defn}[\cite{bubenik15}]\label{def: sub- superlin}
A function $\Omega \colon [0,\infty) \to \Trans P$ is called a \emph{superlinear family} if for all $\epsilon, \epsilon' \in [0,\infty)$
\[\Omega_\epsilon\Omega_{\epsilon'} \leq \Omega_{\epsilon+\epsilon'}.\]
An increasing function $\omega \colon \Trans P \to [0,\infty]$ is called a \emph{sublinear projection} if $\omega_I = 0$, where $I$ is the identity translation on $\mb P$, and for all $\Gamma,\K \in \Trans P$
\[\omega_{\Gamma \K} \leq \omega_\Gamma + \omega_\K.\]
\end{defn}

Note that a superlinear projection is always increasing, since for $\epsilon \leq \epsilon'$
\[\Omega_\epsilon = I\Omega_\epsilon \leq \Omega_{\epsilon'-\epsilon}\Omega_\epsilon \leq \Omega_{\epsilon'}.\]
Hence, a superlinear family is a functor $[0,\infty) \to \Trans P$ and a sublinear projection is functor $\Trans P \to [0,\infty]$.

\begin{defn}[Erosion distance]
If we have a superlinear family $\Omega \colon [0,\infty) \to \Trans P$, we define the \emph{erosion distance w.r.t.} $\Omega$ for maps $f,g \colon \Dgm_{\mb P} \to \mb G$ to be
\[\dist_{E}^{\Omega}(f,g) = \inf \big(\{\epsilon \mid \nabla_{\Omega_\epsilon}f \leq g,\ \nabla_{\Omega_\epsilon}g \leq f\} \cup \{\infty\}\big).\]
If we have a sublinear projection $\omega \colon \Trans P \to [0,\infty]$, we define the \emph{erosion distance w.r.t.} $\omega$ for maps $f,g \colon \Dgm_{\mb P} \to \mb G$ to be
\begin{align*}
\dist_{E}^{\omega}(f,g) = \inf \big(\{\epsilon \mid \exists \Gamma,\K \text{ s.t.~} \omega_\Gamma,\omega_\K \leq \epsilon \text{ and } \nabla_{\Gamma, \K}f \leq g, \nabla_{\K, \Gamma}g \leq f\} \cup \{\infty\}\big).
\end{align*}
If the choice of $\Omega$ or $\omega$ is clear from context, we use a shorthand notation $\dist_E$ for the erosion distance.
\end{defn}

\begin{lemma}\label{lemma: -1 order}
For all $\Gamma, \K \in \Trans P$
\[\Gamma \leq \K \Rightarrow \K^{-1} \leq \Gamma^{-1}.\]
\end{lemma}
\begin{proof}
Let's assume that $\Gamma \leq \K$ and take $a \in \mb P$. Since $\Gamma$ has to be a bijection, we can take $b \in \mb P$ such that $a = \Gamma b$. Now
\begin{align*}
\Gamma b \leq \K b &\Rightarrow K^{-1}\Gamma b \leq b
\\&\Rightarrow \K^{-1} \Gamma b \leq \Gamma^{-1} \Gamma b
\\&\Rightarrow \K^{-1} a \leq \Gamma^{-1} a.
\end{align*}
\end{proof}

\begin{prop}\label{thm: metric}\ 
\begin{enumerate}[\itshape i)]
\item If $\Omega \colon [0,\infty) \to \Trans P$ is a superlinear family, then $\dist_E^{\Omega}$ is an extended pseudo-metric on the set of \emph{decreasing} functions $\Dgm_{\mb P} \to \mb G$.
\item If $\Omega \colon [0,\infty) \to \Trans P$ is linear, i.e.~$\Omega_0 = I$ and $\Omega_a\Omega_b = \Omega_{a+b}$ for all $a,b \in [0,\infty)$, then $\dist_E^{\Omega}$ is an extended pseudo-metric on the set of all functions $\Dgm_{\mb P} \to \mb G$.
\item If $\omega \colon \Trans P \to [0,\infty]$ is a sublinear projection, then $\dist_E^{\omega}$ is an extended pseudo-metric on the set of all functions $\Dgm_{\mb P} \to \mb G$.
\end{enumerate}
\end{prop}
\begin{proof}
It's trivial that $\dist_E^{\Omega}$ and $\dist_E^{\omega}$ are symmetric and non-negative in all cases. Additionally, in cases \emph{ii)} and \emph{iii)} it's clear that $\dist_E(f,f) = 0$ for all maps $f \colon \Dgm_{\mb P} \to \mb G$. If we take a decreasing map $f \colon \Dgm_{\mb P} \to \mb G$, we see that $\nabla_{\Gamma,\K}f \leq f$ for all $\Gamma, \K \in \Trans P$, so in particular $\nabla_{\Omega_0}f \leq f$. This implies that $\dist_E^{\Omega}(f,f) = 0$. All that remains to prove is the triangle inequality.

\begin{enumerate}[\itshape i)]
\item Let $f,g,h \colon \Dgm_{\mb P} \to \mb G$ and $\epsilon, \epsilon' \geq 0$ such that
\[\nabla_{\Omega_\epsilon}f \leq g,\ \nabla_{\Omega_\epsilon}g \leq f,\ \nabla_{\Omega_{\epsilon'}}g \leq h,\ \nabla_{\Omega_{\epsilon'}}h \leq g.\]
By the triangle inequality (Proposition \ref{thm: triangle inequality}) $\nabla_{\Omega_{\epsilon'}\Omega_\epsilon}f \leq h$ and $\nabla_{\Omega_\epsilon\Omega_{\epsilon'}}h \leq f$.

Let's assume that $f$ and $h$ are decreasing and take $(a,b) \in \Dgm_{\mb P}$. We notice that by superlinearity and Lemma \ref{lemma: -1 order} $((\Omega_{\epsilon'}\Omega_{\epsilon})^{-1}a , \Omega_{\epsilon'}\Omega_{\epsilon} b) \leq (\Omega_{\epsilon+\epsilon'}^{-1}a, \Omega_{\epsilon+\epsilon'}b)$. Since $f$ is decreasing
\[\nabla_{\Omega_{\epsilon+\epsilon'}}f(a,b) = f(\Omega_{\epsilon+\epsilon'}^{-1}a, \Omega_{\epsilon+\epsilon'}b) \leq f((\Omega_{\epsilon'}\Omega_{\epsilon})^{-1}a , \Omega_{\epsilon'}\Omega_{\epsilon} b) = \nabla_{\Omega_{\epsilon'}\Omega_\epsilon}f(a,b).\]
Hence $\nabla_{\Omega_{\epsilon+\epsilon'}}f \leq h$. Similarly, we see that $\nabla_{\Omega_{\epsilon+\epsilon'}}h \leq f$.

\item Let $f,g,h$ and $\epsilon, \epsilon'$ be as in the previous case. By the same argument, $\nabla_{\Omega_{\epsilon'}\Omega_\epsilon}f \leq h$ and $\nabla_{\Omega_\epsilon\Omega_{\epsilon'}}h \leq f$. If $\Omega$ is linear, these inequalities give us $\nabla_{\Omega_{\epsilon+\epsilon'}}f \leq h$ and $\nabla_{\Omega_{\epsilon+\epsilon'}}h \leq f$.

\item Let $f,g,h \colon \Dgm_{\mb P} \to \mb G$, $\epsilon, \epsilon' \geq 0$ and $\Gamma, \Gamma', \K, \K' \in \Trans P$ such that $\omega_\Gamma,\omega_\K \leq \epsilon$, $\omega_{\Gamma'},\omega_{\K'} \leq \epsilon'$ and
\[\nabla_{\Gamma, \K}f \leq g,\ \nabla_{\K, \Gamma}g \leq f,\ \nabla_{\Gamma', \K'}g \leq h,\ \nabla_{\K', \Gamma'}h \leq g.\]
By the triangle inequality (Proposition \ref{thm: triangle inequality})
\[\nabla_{\Gamma'\Gamma, \K\K'}f \leq h,\ \nabla_{\K\K', \Gamma'\Gamma}h \leq f,\]
and by sublinearity $\omega_{\Gamma'\Gamma},\omega_{\K\K'} \leq \epsilon + \epsilon'$.
\end{enumerate}
\end{proof}

\subsection{The \texorpdfstring{$L^\infty$}{sup}-distance as an erosion distance}\label{sec: The distance induced by the sup-norm}

As our first example, we consider the erosion distance of level set filtrations of functions $f \colon X \to \Rset$, where $X$ is a fixed set. We show that this is simply the $L^\infty$-distance $\dist_\infty(f,g) = \norm{f-g}_\infty$ of functions $f,g \colon X \to \Rset$.

\begin{defn}
Let $X$ be a set. To every function $f \colon X \to \Rset$ we attach a function
\[F \colon \Dgm_\Rset \to \Set,\ F(a,b) = f^{-1}([a,b]),\]
where $\Set$ is the category of sets. Let $\Omega \colon [0,\infty) \to \Transp{\Rset}$, $\Omega_\epsilon(a) = a + \epsilon$. We define a preorder for $\Set$ by taking the opposite of the natural preorder of sets, i.e.~we set
\[A \leq B \iff A \supseteq B.\]
Now, since these functions $\Dgm_\Rset \to \Set$ are clearly decreasing, we can define the erosion distance $\dist_E$ for functions $f,g \colon X \to \Rset$ by setting
\[\dist_E(f,g) = \dist_E^{\Omega}(F,G).\]
\end{defn}

\begin{thm}\label{thm: L_infty}
For all functions $f,g \colon X \to \Rset$
\[\dist_\infty(f,g) = \dist_E(f,g).\]
\end{thm}
\begin{proof}
Let $f,g \colon X \to \Rset$ and $\epsilon \in [0,\infty)$. Now,
\begin{align*}
d_\infty(f,g) \leq \epsilon &\iff g(x) - \epsilon \leq f(x) \leq g(x) + \epsilon \text{ for all } x \in X
\\&\iff g^{-1}([r,r]) \subseteq f^{-1}([r-\epsilon,r+\epsilon]) \text{ for all } r \in \Rset
\\&\iff g^{-1}([a,b]) \subseteq f^{-1}([a-\epsilon,b+\epsilon]) \text{ for all } a \leq b \in \Rset
\\&\iff g^{-1}([a,b]) \subseteq f^{-1}([a-\epsilon,b+\epsilon]) \text{ for all } a < b \in \Rset
\\&\iff \Omega_\epsilon F \leq G.
\end{align*}
To see the second equivalence, set $r = g(x)$, and to see the $\Leftarrow$ direction of the second to last equivalence, note that
\begin{align*}
g^{-1}([r,r]) &= g^{-1}\big(\bigcap_{i = 1}^\infty [r-\frac{1}{n},r]\big)
\\&= \bigcap_{i = 1}^\infty g^{-1}([r-\frac{1}{n},r])
\\&\subseteq \bigcap_{i = 1}^\infty f^{-1}([r-\frac{1}{n}-\epsilon,r+\epsilon])
\\&= f^{-1}\big(\bigcap_{i = 1}^\infty [r-\frac{1}{n}-\epsilon,r+\epsilon]\big)
\\&= f^{-1}([r-\epsilon,r+\epsilon])
\end{align*}
for all $r \in \Rset$. By symmetry of the first inequality, we get $\dist_\infty(f,g) \leq \epsilon \iff \Omega_\epsilon G \leq F$. Hence,
\[\inf \{\epsilon \mid \dist_\infty(f,g) \leq \epsilon\} = \inf \{\epsilon \mid \Omega_\epsilon G \leq F \text{ and } \Omega_\epsilon F \leq G\},\]
i.e.~$\dist_\infty(f,g) = \dist_E(f,g)$.
\end{proof}

\section{Erosion distance for persistence modules}\label{sec: Erosion distance for persistence modules}

In this section, we specialize the erosion distance for rank invariants of persistence modules $\mb P \to \mb C$, where $\mb P$ is a preordered set and $\mb C$ is an Abelian category with a suitable preorder for its objects. First, in subsection \ref{sec: Preorder induced by the Grothendieck group} we go over the details of the erosion distance of \cite{patel16}, and then in subsection \ref{sec: subsection Erosion distance for persistence modules} we define the distance in full generality.

\subsection{Preorder induced by the Grothendieck group}\label{sec: Preorder induced by the Grothendieck group}

The main contribution of \cite{patel16} is a generalization of persistence diagrams to constructible persistence modules over $\Rset$ with values in a category $\mb C$, where $\mb C$ is an essentially small symmetric monoidal category with images. A persistence module $F \colon \Rset \to \mb C$ is said to be constructible, if there exists a finite set $S = \{s_1, \dots, s_n\} \subseteq \Rset$, where $s_1 < s_2 < \cdots < s_n$, such that
\begin{itemize}
\item for $p \leq q < s_1$ the morphism $F(p \leq q) = \id_e$, where $e \in \mb C$ is the neutral element of the monoidal category,
\item for $s_i \leq p \leq q < s_{i+1}$ the morphism $F(p \leq q)$ is an isomorphism, and
\item for $s_n \leq p \leq q$ the morphism $F(p \leq q)$ is an isomorphism.
\end{itemize}
We denote the set of isomorphism classes of $\mb C$ by $\mathcal J(\mb C)$, and we make $\mathcal J(\mb C)$ into a commutative monoid by setting
\[[A] + [B] = [A \otimes B],\]
for all $A,B \in \mb C$. Now, to every constructible persistence module $F \colon \Rset \to \mb C$ we attach a map
\[dF_{\mathcal A} \colon \Dgm_\Rset \to \mathcal A(\mb C),\ dF_{\mathcal A}(a,b) = [\im F(a < b - \delta)],\]
where $\delta > 0$ is small enough so that $\im F(a < b - \delta') \cong \im F(a < b - \delta)$ for all $0 < \delta' < \delta$, and $\mathcal A(\mb C)$ is the Grothendieck group of $\mb C$ obtained by taking the group completion of $\mathcal J(\mb C)$. If $\mb C$ happens to be Abelian, we consider $\mb C$ to be monoidal by taking the tensor product to be the coproduct $\otimes = \oplus$. Then, we attach a second map to $F$
\[dF_{\mathcal B} \colon \Dgm_\Rset \to \mathcal B(\mb C),\ dF_{\mathcal B}(a,b) = [\im F(a < b - \delta)],\]
where $\mathcal B(\mb C)$ is obtained from $\mathcal A(\mb C)$ by adding relations $[A]+[C]=[B]$ for all exact sequences $\begin{tikzcd}[sep=small] 0 \rar & A \rar & B \rar & C \rar & 0, \end{tikzcd}$ and $\delta > 0$ is again sufficiently small.

Since $F$ is constructible, the maps $dF_{\mathcal A}$ and $dF_{\mathcal B}$ have Möbius inversions (\cite[Theorem 4.1]{patel16}), i.e.~functions $F_{\mathcal A} \colon \Dgm_\Rset \to \mathcal A(\mb C)$ and $F_{\mathcal B} \colon \Dgm_\Rset \to \mathcal B(\mb C)$ with finite support such that
\[\sum_{\mb x \geq \mb a} F_{\mathcal A}(\mb x) = dF_{\mathcal A}(\mb a) \text{ and } \sum_{\mb x \geq \mb a} F_{\mathcal B}(\mb x) = dF_{\mathcal B}(\mb a)\]
for all $\mb a \in \Dgm_\Rset$. These functions $F_{\mathcal A}$ and $F_{\mathcal B}$ are called the type $\mathcal A$ and type $\mathcal B$ persistence diagrams of $F$. Since in this article we always assume that $\mb C$ is Abelian, we will only focus on the type $\mathcal B$ diagrams.

\begin{defn}
We define a preorder for the Grothedieck group $\mathcal B(\mb C)$ by setting
\[x \leq y \iff \text{there exists } A \in \mb C \text{ such that } x + [A] = y.\]
This gives a preorder for $\mb C$
\[A \leq B \iff [A] \leq [B].\]
The type $\mathcal B$ persistence diagrams of constructible persistence modules are preordered by setting for all constructible $F,G \colon \Rset \to \mb C$
\[F_{\mathcal B} \preceq G_{\mathcal B} \iff \sum_{\mb x \geq \mb a} F_{\mathcal B}(\mb x) \leq \sum_{\mb x \geq \mb a} G_{\mathcal B}(\mb x) \text{ for all } \mb a \in \Dgm_\Rset.\]
Since $F_{\mathcal B}$ and $G_{\mathcal B}$ are Möbius inversions of $dF_{\mathcal B}$ and $dG_{\mathcal B}$, this is equivalent to
\[dF_{\mathcal B}(\mb a) \leq dG_{\mathcal B}(\mb a) \text{ for all } \mb a \in \Dgm_\Rset.\]
\end{defn}

\begin{defn}
The erosion distance between type $\mathcal B$ persistence diagrams of constructible persistence modules $F,G$ is
\[\dist_E(F,G) = \inf \{\epsilon \geq 0 \mid \nabla_{\Omega_\epsilon} F_{\mathcal B} \preceq G_{\mathcal B} \text{ and } \nabla_{\Omega_\epsilon} G_{\mathcal B} \preceq F_{\mathcal B}\},\]
where $\Omega$ is the usual superlinear family of $\Rset$, $\Omega_\epsilon(a) = a + \epsilon$.
\end{defn}

Once again, using the fact that $F_{\mathcal B}$ and $G_{\mathcal B}$ are Möbius inversions, these inequalities are equivalent to
\[\nabla_{\Omega_\epsilon} dF_{\mathcal B} \leq dG_{\mathcal B} \text{ and } \nabla_{\Omega_\epsilon} dG_{\mathcal B} \leq dF_{\mathcal B},\]
where the inequalities are pointwise inequalities of functions, i.e.~$\nabla_{\Omega_\epsilon} dF_{\mathcal B} \leq dG_{\mathcal B} \iff \nabla_{\Omega_\epsilon} dF_{\mathcal B}(\mb x) \leq dG_{\mathcal B}(\mb x)$ for all $\mb x \in \Dgm_\Rset$.

This way of getting an erosion distance between persistence modules doesn't generalize to arbitrary preordered sets $\mb P$ since we need the $\delta$ in the definition of $dF_{\mathcal B}$. Fortunately, forgetting the $\delta$ in the definition turns out to give the same distance as the next proposition and corollary show.

\begin{prop}
Let $F, G \colon \Rset \to \mb C$ be constructible persistence modules and $\epsilon \in [0,\infty)$. Define
\[\mathcal F \colon \Dgm_\Rset \to \mb C,\ (a,b) \mapsto \im F(a < b),\]
and
\[\mathcal G \colon \Dgm_\Rset \to \mb C,\ (a,b) \mapsto \im G(a < b).\]
Now
\[\nabla_{\Omega_\epsilon} \mathcal F \leq \mathcal G \iff \nabla_{\Omega_\epsilon} dF_{\mathcal B} \leq dG_{\mathcal B} \iff \nabla_{\Omega_\epsilon} F_{\mathcal B} \preceq G_{\mathcal B}.\]
\end{prop}
\begin{proof}
The right-hand equivalence follows directly from the definition of the rightmost inequality and by the definition of the Möbius inversion. Let $(a,b) \in \Dgm_\Rset$ and let's first assume that $\nabla_{\Omega_\epsilon} \mathcal F \leq \mathcal G$. Now, by the definition of $dF_{\mathcal B}$ and $dG_{\mathcal B}$ there exists $\delta > 0$ such that
\[\nabla_{\Omega_\epsilon} dF_{\mathcal B}(a,b) = [\im F(a - \epsilon < b + \epsilon - \delta)]\]
and
\[dG_{\mathcal B}(a,b) = [\im G(a < b - \delta)].\]
Hence
\begin{align*}
\nabla_{\Omega_\epsilon} dF_{\mathcal B}(a,b) \leq dG_{\mathcal B}(a,b) &\iff [\im F(a - \epsilon < b + \epsilon - \delta)] \leq [\im G(a < b - \delta)]
\\&\iff \im F(a - \epsilon < b + \epsilon - \delta) \leq \im G(a < b - \delta)
\\&\iff \nabla_{\Omega_\epsilon} \mathcal F(a,b-\delta) \leq \mathcal G(a,b-\delta).
\end{align*}
The last inequality holds by assumption, so $\nabla_{\Omega_\epsilon} dF_{\mathcal B} \leq dG_{\mathcal B}$.

Now, let's assume that $\nabla_{\Omega_\epsilon} dF_{\mathcal B} \leq dG_{\mathcal B}$ and let $(a,b) \in \Dgm_\Rset$. Since $F$ and $G$ are constructible, there exists a small enough $\delta > 0$ such that for all $0 < \delta' \leq \delta$
\[F(b+\epsilon < b+\epsilon+\delta') \colon F(b+\epsilon) \cong F(b+\epsilon+\delta')\]
and
\[G(b < b+\delta') \colon G(b) \cong G(b+\delta'),\]
i.e.~the morphisms are isomorphisms. Hence
\[\nabla_{\Omega_\epsilon} dF_{\mathcal B}(a,b+\delta) = [\im F(a-\epsilon < b+\epsilon)]\]
and
\[dG_{\mathcal B}(a,b+\delta) = [\im G(a < b)].\]
Now
\begin{align*}
\nabla_{\Omega_\epsilon} \mathcal F(a,b) \leq \mathcal G(a,b) &\iff \im F(a-\epsilon < b+\epsilon) \leq \im G(a < b)
\\&\iff [\im F(a-\epsilon < b+\epsilon)] \leq [\im G(a < b)]
\\&\iff \nabla_{\Omega_\epsilon} dF_{\mathcal B}(a,b+\delta) \leq dG_{\mathcal B}(a,b+\delta).
\end{align*}
Again, the last inequality holds by assumption, so $\nabla_{\Omega_\epsilon} \mathcal F \leq \mathcal G$.
\end{proof}

\begin{cor}
Let $F$ and $G$ be constructible persistence modules. Then
\[\dist_E^{\Omega}(\mathcal F, \mathcal G) = \dist_E(F_{\mathcal B}, G_{\mathcal B}).\]
\end{cor}

\subsection{Erosion distance for persistence modules}\label{sec: subsection Erosion distance for persistence modules}

In this subsection we extend the idea of the previous subsection for persistence modules $\mb P \to \mb C$, where $\mb P$ is a preordered set and $\mb C$ is an Abelian category with a preorder for its objects.

If we have a sublinear projection $\omega$, or a linear family $\Omega$, Proposition \ref{thm: metric} shows that this information is enough to make the erosion distance an extended pseudo-metric for functions $\Dgm_{\mb P} \to \mb C$. However, if we have a superlinear family that is not linear, we need to restrict ourselves to decreasing functions. To make sure that the functions induced by persistence modules, i.e. the rank invariants, are indeed decreasing, we first need to consider which preorders of $\mb C$ are suitable. A natural idea is to require objects to be larger than their subobjects and quotients, and this turns out to be enough; preorders that satisfy this condition will be said to respect mono- and epimorphisms.

\begin{lemma}\label{lem: respects monos and epis}
Let $\mb C$ be an Abelian category equipped with a preorder $\leq$ for its objects such that for all $A,B \in \mb C$
\[A \hookrightarrow B \Rightarrow A \leq B\]
and
\[A \twoheadrightarrow B \Rightarrow A \geq B.\]
Then, for all morphisms $f \colon A \to B$
\begin{enumerate}[\itshape i)]
\item $\ker f \leq A$ and $\coker f \leq B$,
\item $\im f \leq A, B$,
\item if $f$ is an isomorphism, then $A \leq B$ and $B \leq A$.
\end{enumerate}
Additionally, every preorder that satisfies condition i) also satisfies
\[A \hookrightarrow B \Rightarrow A \leq B\]
and
\[A \twoheadrightarrow B \Rightarrow A \geq B.\]
\end{lemma}
\begin{proof}
Cases \emph{i)}-\emph{iii)} are trivial. The last remark follows from the fact that in an Abelian category every monomorphism is a kernel morphism and every epimorphism is a cokernel morphism.
\end{proof}

\begin{defn}
A preorder of an Abelian category $\mb C$ that satisfies the conditions in the previous lemma is said to \emph{respect mono- end epimorphisms}. Throughout the rest of this paper, $\mb C$ is an Abelian category equipped with a preorder that respects mono- and epimorphisms unless otherwise stated.
\end{defn}

\begin{lemma}\label{lemma: im order}
Let $f \colon A \to B$, $g \colon A' \to A$, $h \colon B \to B'$ be morphisms in $\mb C$ and denote $f' = hfg$. Then
\[\im f' \leq \im f.\]
\end{lemma}
\begin{proof}
By using the universal property of images we can construct the following commutative diagram
\[\begin{tikzcd}[column sep=small]
A' \arrow[to path= |- (\tikztotarget)]{dddrrr}\ar{dr}{g}\ar{rrrr}{f'} &&&& B' \\
& A \ar{d} \ar{rr}{f} && B \ar{ur}{h} \\
& \im f \arrow[hook]{urr} \arrow[twoheadrightarrow]{rr} && \im\big(\im f \to B'\big) \arrow[hook]{uur} \\
&&& \im f' \arrow[hook]{u}{\exists} \arrow[hook, to path= -| (\tikztotarget)]{ruuu}
\end{tikzcd}\]
Hence
\[\im f' \leq \im\big(\im f \to B'\big) \leq \im f.\]
\end{proof}

\begin{defn}[Rank invariant]\label{def: persmod erosion dist}
To every persistence module $F \colon \mb P \to \mb C$, i.e.~a functor, we attach a map $\mathcal F \colon \Dgm_{\mb P} \to \mb C$ by setting for each $(a,b) \in \Dgm_{\mb P}$
\[\mathcal F(a,b) = \im F(a < b).\]
We call this map the \emph{rank invariant} of $F$. In addition, let $\Omega \colon [0,\infty) \to \Trans P$ be a superlinear family or $\omega \colon \Trans P \to [0,\infty]$ be a sublinear projection. We define the \emph{erosion distance} $\dist_E$ of a pair of persistence modules $F,G \colon \mb P \to \mb C$ to be
\[\dist_E^\Omega(F,G) = \dist_E^{\Omega}(\mathcal F,\mathcal G)\]
or
\[\dist_E^\omega(F,G) = \dist_E^{\omega}(\mathcal F,\mathcal G).\]
\end{defn}

\begin{prop}\label{thm: decreasing}
For every persistence module $F \colon \mb P \to \mb C$ the map $\mathcal F$ is decreasing.
\end{prop}
\begin{proof}
Let $(a,b) \leq (a',b')$, i.e.~$a' \leq a < b \leq b'$. Since
\[F(a' < b') = F(b \leq b')F(a < b)F(a' \leq a),\]
Lemma \ref{lemma: im order} says that
\begin{align*}
\mathcal F(a',b') &= \im F(a' < b')
\\&\leq \im F(a < b) 
\\&= \mathcal F(a,b).
\end{align*}
\end{proof}

\begin{cor}\label{cor: metric for persistence modules}
The erosion distances for persistence modules $\dist_E^{\Omega}$ and $\dist_E^{\omega}$ are extended pseudo-metrics.
\end{cor}
\begin{proof}
The claim follows directly from Proposition \ref{thm: decreasing} and Proposition \ref{thm: metric} \emph{i)} and \emph{iii)}.
\end{proof}

Now that we have shown that the erosion distances are extended pseudo-metrics, we'll consider stability with respect to the interleaving distance introduced in \cite{bubenik15}.

\begin{defn}[{\cite{bubenik15}}]
Let $\Gamma, \K \in \Trans P$ and let $F,G \colon \mb P \to \mb C$ be persistence modules. A $(\Gamma, \K)$-\emph{interleaving} between $F$ and $G$ is a pair of natural transformations $(\varphi, \psi)$
\[\varphi \colon F \To G\Gamma,\ \psi \colon G \To F\K,\]
such that the following diagrams commute:
\[\begin{tikzcd}[column sep=small]
F \arrow[Rightarrow]{dr}{\varphi} \arrow[Rightarrow]{rr} && F\K\Gamma \\
& G\Gamma \arrow[Rightarrow]{ur}{\psi}
\end{tikzcd}
\quad
\begin{tikzcd}[column sep=small]
G \arrow[Rightarrow]{dr}{\psi}\arrow[Rightarrow]{rr} && G\Gamma\K \\
& F\K \arrow[Rightarrow]{ur}{\varphi}
\end{tikzcd}\]
We say that $F$ and $G$ are $\epsilon$-\emph{interleaved with respect to} $\Omega$ if they are $(\Omega_\epsilon, \Omega_\epsilon)$-interleaved, and similarly that they are $\epsilon$-\emph{interleaved with respect to} $\omega$ if there exist $\Gamma, \K \in \Trans P$ such that $\omega_\Gamma,\omega_\K \leq \epsilon$ and $F$ and $G$ are $(\Gamma,\K)$-interleaved.

We define the \emph{interleaving distances} $\dist_I^{\Omega}$ and $\dist_I^{\omega}$ by setting
\[\dist_I^{\Omega}(F,G) = \inf \big(\{\epsilon \mid F \text{ and } G \text{ are } \epsilon \text{-interleaved w.r.t. } \Omega\} \cup \{\infty\}\big)\]
\[\dist_I^{\omega}(F,G) = \inf \big(\{\epsilon \mid F \text{ and } G \text{ are } \epsilon \text{-interleaved w.r.t. } \omega\} \cup \{\infty\}\big).\]
\end{defn}

\begin{remark}
Note that since in \cite{bubenik15} a translation of $\mb P$ is not required to be an automorphism, and instead is only required to be an endofunctor of $\mb P$ with a natural transformation from the identity functor, our definition of the interleaving distance is slightly different. Specifically, our definition of $\dist_I^{\Omega}$ is precisely the same, but for us the choice of $\Omega$ is more restricted, and our definition of $\dist_I^{\omega}$ may be larger than the distance defined in \cite{bubenik15}.
\end{remark}

\begin{thm}[Stability of the erosion distance]\label{thm: stability of erosion wrt interleaving}
Let $F,G \colon \mb P \to \mb C$ be persistence modules. Then
\[\dist_E(F,G) \leq \dist_I(F,G),\]
where either $\dist_E = \dist_E^{\Omega}$ and $\dist_I = \dist_I^{\Omega}$, or $\dist_E = \dist_E^{\omega}$ and $\dist_I = \dist_I^{\omega}$.
\end{thm}
\begin{proof}
To prove the claim in both cases, it is enough to show that if we have a $(\Gamma,\K)$-interleaving between $F$ and $G$, then
\[\nabla_{\Gamma,\K} \mathcal F \leq \mathcal G,\ \nabla_{\K,\Gamma} \mathcal G \leq \mathcal F.\]
Let $(\varphi, \psi)$ be a $(\Gamma,\K)$-interleaving between $F$ and $G$. For every $(a,b) \in \Dgm_{\mb P}$ we get a commutative diagram
\[\begin{tikzcd}[column sep=small]
F(\Gamma^{-1} a) \ar{dr}{\varphi}\ar{rr} && F(\K a) \rar & F(\K b) \\
& G(a) \ar{ur}{\psi} \rar & G(b) \ar{ur}{\psi}
\end{tikzcd}\]
This shows that
\[F(\Gamma^{-1} a < \K b) = \psi_{b} \circ G(a < b) \circ \varphi_{\Gamma^{-1} a},\]
and then by Lemma \ref{lemma: im order}
\begin{align*}
\nabla_{\Gamma,\K}\mathcal F(a,b) &= \im F(\Gamma^{-1}a, \K b)
\\&\leq \im G(a < b)
\\&= \mathcal G(a,b).
\end{align*}
Hence $\nabla_{\Gamma,\K}\mathcal F \leq \mathcal G$. Similarly, we can show that $\nabla_{\K, \Gamma}\mathcal G \leq \mathcal F$.
\end{proof}

\section{Minimal preorders and the natural pseudo-distance}\label{sec: Minimal preorders and the natural pseudo-distance}

In this section we show how the distance $\dist_T$ of \cite{frosini13} is obtained as a special case of the erosion distance. The distance $\dist_T$ is an extended pseudo-metric for continuous functions $\varphi \colon X \to \Rset^n$ for some fixed $n \in \Zset_+$ and any topological space $X$. We also show that the recipe for getting a preorder of $\Ab$ that is used in \cite{frosini13} can be used in an arbitrary Abelian category $\mb C$, and that it gives the minimal preorder of $\mb C$ that respects mono- and epimorphisms.

Before looking at the relationship between our general erosion distance and the distance $\dist_T$, we start with some definitions and propositions to help us declutter the definition of $\dist_T$ and understand the preorder of $\Ab$ that is implicitly defined in \cite{frosini13}.

\begin{defn}
Let $\Omega \colon [0,\infty) \to \Transp{\Rset^n}$, $\Omega_\epsilon(\mb a) = \mb a + \bm{\epsilon}$, where $\bm{\epsilon} = (\epsilon, \dots, \epsilon)$. We denote the category of all Abelian groups by $\Ab$. We define a preorder for $\Ab$ by setting for every $A,B \in \Ab$
\[A \leq B \iff \text{there exists a subgroup } B' \subseteq B \text{ and an epimorphism } B' \twoheadrightarrow A.\]
\end{defn}

\begin{prop}
The relation $\leq$ defined above is a preorder for $\Ab$.
\end{prop}
\begin{proof}
Reflexivity is trivial. Let $A \leq B \leq C$, i.e.~there exist subgroups $B' \subseteq B$ and $C' \subseteq C$ such that $B' \twoheadrightarrow A$ and $f \colon C' \twoheadrightarrow B$. We define $C'' = f^{-1}(B')$. Now, clearly $C'' \twoheadrightarrow A$, so $A \leq C$.
\end{proof}

This preorder clearly respects mono- and epimorphisms. Hence, we get a stable extended pseudo-metric $\dist_E^{\Omega}$.
\begin{defn}
Let $F,G \colon \Rset^n \to \Ab$ be persistence modules. The erosion distance of $F$ and $G$ is
\[\dist_E^{\Omega}(F,G) = \dist_E^{\Omega}(\mathcal F,\mathcal G).\]
\end{defn}

This preorder turns out to be minimal among all preorders that respect mono- and epimorhisms in any Abelian category, as long as it actually defines a preorder. Even if it doesn't define a preorder, its transitive closure is the minimal preorder.

\begin{prop}\label{thm: min preorder}
Let $\leq$ be any preorder for the Abelian category $\mb C$ such that $\leq$ respects mono- and epimorphisms. Define a relation $R \subset \mb C \times \mb C$ by setting
\[aRb \iff \text{ there exists a diagram } a \twoheadleftarrow b' \hookrightarrow b\]
for all $a,b \in \mb C$. Then, for all $a,b \in \mb C$
\[aRb \Rightarrow a \leq b.\]
\end{prop}
\begin{proof}
Let $a,b \in \mb C$ such that $aRb$, i.e.~there exists $b' \in \mb C$ such that
\[a \twoheadleftarrow b' \hookrightarrow b.\]
Since $\leq$ respects mono- and epimorphisms, $a \leq b' \leq b$.
\end{proof}

\begin{cor}
Let $R$ be as in Proposition \ref{thm: min preorder} and let $\leq$ be its transitive closure. Then $\leq$ is minimal among all preorders of $\mb C$ that respect mono- and epimorphisms, i.e.~if $\preceq$ is another preorder that respects mono- and epimorphisms, then for all $A,B \in \mb C$
\[A \leq B \Rightarrow A \preceq B.\]
In addition, let $\mb P$ be a preordered set and fix a superlinear family $\Omega$ (resp. a sublinear projection $\omega$) of $\mb P$. Then, the erosion distance with respect to $\leq$ and $\Omega$ (resp. $\omega$) is maximal among all erosion distances of functions $\Dgm_{\mb P} \to \mb C$ with respect to $\Omega$ (resp. $\omega$).
\end{cor}

To define the erosion distance between continuous functions $\varphi \colon X \to \Rset^n$, where we allow the space $X$ to vary, we first take the sublevelset filtration of $X$ induced by $\varphi$ which gives us a functor $\Rset^n \to \Top$, then apply singular homology which gives us a functor $\Rset^n \to \Ab$. Now we can apply the erosion distance $\dist_E^{\Omega}$ for functors $\Rset^n \to \Ab$ as defined in Definition \ref{def: persmod erosion dist}.

With more detail: we take continuous functions $\varphi \colon X \to \Rset^n$, $\psi \colon Y \to \Rset^n$ where $X$ and $Y$ are topological spaces. For all $\mb a < \mb b \in \Rset^n$, we set
\[H_k^{X,\varphi}(\mb a,\mb b) = \im H_k\big(X\langle \varphi \leq \mb a \rangle \subseteq X\langle \varphi \leq \mb b \rangle\big),\]
where $X\langle \varphi \leq \mb c \rangle := \{x \in X \mid \varphi(x) \leq \mb c\}$ for all $\mb c \in \Rset^n$, and $H_k$ is the $k$-th singular homology with coefficients in an Abelian group $A$. Similarly, we set
\[H_k^{Y,\psi}(\mb a,\mb b) = \im H_k\big(Y\langle \psi \leq \mb a \rangle \subseteq Y\langle \psi \leq \mb b \rangle\big).\]
Note that $H_k^{X,\varphi}$ and $H_k^{Y,\psi}$ are the rank invariants of the sublevel set persistent homologies of $\varphi$ and $\psi$, and especially they are maps $\Dgm_{\Rset^n} \to \Ab$.

\begin{defn}
The erosion distance between $\varphi \colon X \to \Rset^n$ and $\psi \colon Y \to \Rset^n$ is
\[\dist_E^{\Omega}(\varphi,\psi) = \dist_E^{\Omega}(H_k^{X,\varphi},H_k^{Y,\psi}).\]
\end{defn}

The distance $\dist_T$ is defined similarly with a subtle difference: set $\Dgm'_{\Rset^n} := \{(\mb a,\mb b) \in \Rset^n \mid a_i < b_i \text{ for each } i = 1, \dots, n\} \subseteq \Dgm_{\Rset^n}$ and define functions
\[\hat \varphi \colon \Dgm'_{\Rset^n} \to \Ab,\ \hat \varphi (\mb a,\mb b) = H_k^{X,\varphi}(\mb a,\mb b),\]
and
\[\hat \psi \colon \Dgm'_{\Rset^n} \to \Ab,\ \hat \psi (\mb a,\mb b) = H_k^{Y,\psi}(\mb a,\mb b),\]
i.e.~$\hat \varphi$ and $\hat \psi$ are restrictions of $H_k^{X,\varphi}$ and $H_k^{Y,\psi}$ to $\Dgm'_{\Rset^n}$.
\begin{defn}
The distance $\dist_T$ between $\varphi$ and $\psi$ is
\[\dist_T(\varphi,\psi) = \inf \big(\{\epsilon \in [0,\infty) \mid \nabla_{\Omega_\epsilon} \hat \varphi \leq \hat \psi \text{ and } \nabla_{\Omega_\epsilon} \hat \psi \leq \hat \varphi\} \cup \{\infty\} \big),\]
where $\nabla_{\Omega_\epsilon} \hat \varphi$ and $\nabla_{\Omega_\epsilon} \hat \psi$ are restricted to $\Dgm'_{\Rset^n}$.
\end{defn}

The only difference between $\dist_T$ and $\dist_E^{\Omega}$ is that the inequalities considered in the definition of $\dist_E^{\Omega}$ are of functions defined over $\Dgm_{\Rset^n}$ while in the definition of $\dist_T$ the inequalities are of \emph{the same functions} restricted to $\Dgm'_{\Rset^n} \subset \Dgm_{\Rset^n}$. Hence,
\[\dist_T(\varphi,\psi) \leq \dist_E^{\Omega}(\varphi,\psi).\]
The converse inequality actually holds as well, as the next proposition implies.

\begin{prop}
Let $f,g \colon \Dgm_{\Rset^n} \to \mb G$ be decreasing functions and $\epsilon > 0$. If $\nabla_{\Omega_\epsilon}f |_{\Dgm'_{\Rset^n}} \leq g |_{\Dgm'_{\Rset^n}}$, then $\nabla_{\Omega_{\epsilon'}} f \leq g$ for all $\epsilon' > \epsilon$.
\end{prop}
\begin{proof}
Let $\epsilon' > \epsilon$ and take any $(\mb a, \mb b) \in \Dgm_{\Rset^n}$. Note that since $\epsilon' - \epsilon > 0$, we have
\[(\mb a - (\bm \epsilon' - \bm \epsilon), \mb b + (\bm \epsilon' - \bm \epsilon)) \in \Dgm'_{\Rset^n}.\]
Now, using the inequality $\nabla_{\Omega_\epsilon}f |_{\Dgm'_{\Rset^n}} \leq g |_{\Dgm'_{\Rset^n}}$ and the fact that $g$ is decreasing,
\begin{align*}
\nabla_{\Omega_{\epsilon'}} f (\mb a,\mb b) &= f(\mb a- \bm \epsilon', \mb b + \bm \epsilon')
\\&= \nabla_{\Omega_\epsilon} f(\mb a - (\bm \epsilon' - \bm \epsilon), \mb b + (\bm \epsilon' - \bm \epsilon))
\\&\leq g(\mb a - (\bm \epsilon' - \bm \epsilon), \mb b + (\bm \epsilon' - \bm \epsilon))
\\&\leq g(\mb a,\mb b).
\end{align*}
Hence $\nabla_{\Omega_{\epsilon'}} f \leq g$.
\end{proof}

As noted before the previous proposition, we now see that
\begin{equation}\label{eq: d_T = d_E}
\dist_T(\varphi,\psi) = \dist_E^{\Omega}(\varphi,\psi).
\end{equation}

One of the central results of \cite{frosini13} is Theorem 2.9 that states that $\dist_T$ gives a lower bound for the natural pseudo-distance. The natural pseudo-distance is a dissimilarity measure between size pairs, i.e.~topological spaces equipped with continuous $\Rset^n$-valued functions. It measures how close we can get two functions corresponding to two size pairs, with respect to the $L^\infty$-distance, by changing the base space of one of the functions to the base space of the other function by a homeomorphism.

\begin{defn}
A \emph{size pair} $(X, \varphi)$ consists of a topological space $X$ and a continuous function $\varphi \to \Rset^n$. \emph{The natural pseudo-distance} between two size pairs $(X, \varphi)$, $(Y,\psi)$ is
\[\dist_{\mathit{NP}}(\varphi,\psi) = \inf_{h \in \Homeo(X,Y)} \norm{\varphi - \psi \circ h}_\infty,\]
where $\Homeo(X,Y)$ is the set of homoeomorphisms from $X$ to $Y$, $\Rset^n$ is equipped with the max-norm $\norm{\mb x} = \max_{i=1,\dots,n} \abs{x_i}$, and $\norm{f}_\infty = \sup_{x \in X}\norm{f(x)}$ is the sup-norm.
\end{defn}

For more on the natural pseudo-distance, see e.g. \cite{donatini04}, \cite{donatini07}, \cite{donatini09}, and to see how the natural pseudo-distance can be interpreted as an interleaving distance, see \cite[Section~3.3]{desilva17}.

\begin{thm}[Lower bound for the natural pseudo-distance]\label{thm: lower bound nat pseudodist}
Let $X$ and $Y$ be homeomorphic topological spaces equipped with continuous maps $\varphi \colon X \to \Rset^n$ and $\psi \colon Y \to \Rset^n$. Let $H \colon \Top \to \mb C$ be a functor. The functions $\varphi$ and $\psi$ induce functors
\[\varphi^\leq, \psi^\leq \colon \Rset^n \to \Top\]
by taking sublevel sets. Concretely,
\[\varphi^\leq (\mb a) = X\langle \varphi \leq \mb a \rangle = \{x \in X \mid \varphi(x) \leq \mb a\},\]
\[\psi^\leq (\mb a) = Y\langle \psi \leq \mb a \rangle = \{y \in Y \mid \psi(y) \leq \mb a\},\]
for all $\mb a \in \Rset^n$, and morphisms $\varphi^\leq (\mb a \leq \mb b)$ and $\psi^\leq (\mb a \leq \mb b)$ are simply inclusions for all $\mb a \leq \mb b \in \Rset^n$. Remember that $\mb C$ is equipped with a preorder that respects mono- and epimorphisms. Now
\[\dist_E^{\Omega}(H\varphi^\leq,H\psi^\leq) \leq \dist_{\mathit{NP}}(\varphi,\psi).\]
\end{thm}
\begin{proof}
To simplify notation, let's denote the rank invariants of $H\varphi^\leq$ and $H\psi^\leq$ simply by $H\varphi^\leq$ and $H\psi^\leq$.

Let $h \colon X \to Y$ be a homeomorphism and set $\epsilon = \norm{\varphi - \psi \circ h}_\infty$. We first note that $\norm{\varphi \circ h^{-1} - \psi}_\infty = \epsilon$. We need to show that $\nabla_{\Omega_\epsilon} H\varphi^\leq \leq H\psi^\leq$ and $\nabla_{\Omega_\epsilon} H\psi^\leq \leq H\varphi^\leq$. We'll show only the former inequality since the latter can be shown in exactly the same way. Let $(\mb a,\mb b) \in \Dgm_{\Rset^n}$. Note that $h$ and $h^{-1}$ can be restricted to maps
\[X\langle \varphi \leq \mb a-\bm \epsilon \rangle \xrightarrow{h} Y\langle \psi \leq \mb a \rangle\]
and
\[Y\langle \psi \leq \mb b \rangle \xrightarrow{h^{-1}} X\langle \varphi \leq \mb b + \bm \epsilon \rangle,\]
since
\begin{align*}
\varphi(x) \leq \mb a - \bm \epsilon &\Rightarrow \varphi_i(x) \leq a_i - \epsilon \ \forall i = 1, \dots, n
\\&\Rightarrow \psi_i(h(x)) = \psi_i(h(x)) - \varphi_i(x) + \varphi_i(x) \leq \epsilon + a_i - \epsilon = a_i\ \forall i
\\&\Rightarrow \psi(h(x)) \leq \mb a
\end{align*}
and similarly $\psi(x) \leq \mb b \Rightarrow \varphi(h^{-1}(x)) \leq \mb b + \bm \epsilon$. Then, note that the composition of maps
\[X\langle \varphi \leq \mb a - \bm \epsilon \rangle \xrightarrow{h} Y\langle \psi \leq \mb a \rangle \subseteq Y\langle \psi \leq \mb b \rangle \xrightarrow{h^{-1}} X\langle \varphi \leq \mb b + \bm \epsilon \rangle\]
is simply the inclusion
\[X\langle \varphi \leq \mb a - \bm \epsilon \rangle \subseteq X\langle \varphi \leq \mb b + \bm \epsilon \rangle.\]
Since $H$ is a functor, we get a commutative diagram
\[\begin{tikzcd}
H(X\langle \varphi \leq \mb a - \bm \epsilon \rangle) \dar{H(h)} \rar{H} & H(X\langle \varphi \leq \mb b + \bm \epsilon \rangle) \\
H(Y\langle \psi \leq \mb a \rangle) \rar{H} & H(Y\langle \psi \leq \mb b \rangle) \uar{H(h^{-1})}
\end{tikzcd}\]
The image of the upper horizontal map is $\nabla_{\Omega_\epsilon} H\varphi^\leq(\mb a,\mb b) = H\varphi^\leq(\mb a - \bm \epsilon, \mb b + \bm \epsilon)$ and the image of the lower horizontal map is $H\psi^\leq(\mb a,\mb b)$. By Lemma \ref{lemma: im order}
$\nabla_{\Omega_\epsilon} H\varphi^\leq(\mb a,\mb b) \leq H\psi^\leq(\mb a,\mb b)$ and further $\nabla_{\Omega_\epsilon} H\varphi^\leq \leq H\psi^\leq$.
\end{proof}

\begin{rem}
Note that the previous proof can be modified to give a proof for the fact that the interleaving distance gives a lower bound for the natural pseudodistance, i.e.
\[\dist_I^{\Omega}(H\varphi^\leq,H\psi^\leq) \leq \dist_{\mathit{NP}}(\varphi,\psi).\]
This is done by noting that the last commutative diagram shows that $H(h)$ and $H(h^{-1})$ give an $\Omega_\epsilon$-interleaving between the functors $H\varphi^\leq$ and $H\psi^\leq$. Then, since we already showed that the erosion distance is smaller than the interleaving distance (Theorem \ref{thm: stability of erosion wrt interleaving}) we get the theorem.

A third way to prove the theorem is to use the fact the persistent sublevel set homology of a size pair remains invariant when we change the base space with a homeomorphism (see e.g. \cite[Appendix A]{frosini16}). This fact, combined with the classical stability theorem of persistent homology, shows that the interleaving distance gives a lower bound for the natural pseudodistance. Then, we can again use Theorem \ref{thm: stability of erosion wrt interleaving} to show the previous theorem.
\end{rem}

\begin{cor}[{\cite[Theorem 2.9]{frosini13}}]
Let $(X,\varphi)$ and $(Y,\psi)$ be as in the previous theorem. Then
\[\dist_T(\varphi,\psi) = \dist_E^{\Omega}(\varphi,\psi) \leq \dist_{\mathit{NP}}(\varphi,\psi).\]
\end{cor}
\begin{proof}
The equality $\dist_T(\varphi,\psi) = \dist_E^{\Omega}(\varphi,\psi)$ is precisely equation \ref{eq: d_T = d_E}, and the inequality is the conclusion of Theorem \ref{thm: lower bound nat pseudodist}.
\end{proof}

\section{Adjunction relation}\label{sec: Adjunction relation}

If we have a superlinear family and a sublinear projection for $\mb P$, a natural question is, when are the two erosion distances equal. In \cite{bubenik15} it was shown that if the family and the projection satisfy the so-called adjunction relation, then the two interleaving distances are equal. The same argument can be applied to show the equality of the two interleaving distances in our case where the set of translations is smaller. In this section we show that the same conclusion holds for the two erosion distances.

Remember that $\mb P$ is a preordered set, $\mb G$ is a preordered class, and $\mb C$ is an Abelian category equipped with a preorder that respects mono- and epimorphisms.

\begin{defn}[\cite{bubenik15}]
Let $\Omega \colon [0,\infty) \to \Trans P$ be a superlinear family and $\omega \colon \Trans P \to [0,\infty]$ be a sublinear projection. We say that $\omega$ and $\Omega$ satisfy the \emph{adjuction relation}, if for all $\epsilon \in [0, \infty)$ and $\Gamma \in \Trans P$
\[\omega_\Gamma \leq \epsilon \iff \Gamma \leq \Omega_\epsilon.\]
We also say that $\omega$ and $\Omega$ are an \emph{adjoint pair}. We denote this relation by $\omega \dashv \Omega$.
\end{defn}

Note that since all the categories in the definition are thin, the relation is almost precisely the adjunction relation of functors, with the only difference being that the domain of $\Omega$ is not equal to the codomain of $\omega$. If $\Trans P$ has a maximum, i.e.~a translation that is larger than every other translation, then we can extend the domain of $\Omega$ to $[0,\infty]$ and we'll get an adjoint pair of functors.

Before showing that the erosion distances of an adjoint pair $\omega \dashv \Omega$ are equal, we'll give a description of how $\Omega$ (resp. $\omega$) determines $\omega$ (resp. $\Omega$).

\begin{thm}[{\cite[2.19]{puuska16}}]\footnote{Note that the same proof shows the theorem in the setting of \cite{bubenik15} where the set of translations is larger.}
\ 
\begin{enumerate}[\itshape i)]
\item Let $\Omega \colon [0,\infty) \to \Trans P$ be a superlinear family. There exists a sublinear projection $\omega \colon \Trans P \to [0,\infty]$ such that $\omega \dashv \Omega$, if and only if for all $\Gamma \in \Trans P$ the set
\[\{\epsilon \in [0,\infty) \mid \Gamma \leq \Omega_\epsilon\} \cup \{ \infty \}\]
has a minimum. If $\omega$ exists, then $\omega_\Gamma$ is the minimum of the above set for every $\Gamma \in \Trans P$.

\item Let $\omega \colon \Trans P \to [0,\infty]$ be a sublinear projection. There exists a superlinear family $\Omega \colon [0,\infty) \to \Trans P$ such that $\omega \dashv \Omega$, if and only if for all $\epsilon \in [0, \infty)$ the set
\[\{\Gamma \in \Trans P \mid \omega_\Gamma \leq \epsilon\}\]
has a maximum.\footnote{By a maximum we mean an element of the set that is larger than every other element of the set. Since $\mb P$ is only preordered, $\Trans P$ is also only preordered, and so the maximums might not be unique.} If $\Omega$ exists, then $\Omega_\epsilon$ is one of the possibly many maximums of the above set for every $\epsilon \in [0,\infty)$.
\end{enumerate}
\end{thm}
\begin{proof}
\ 
\begin{enumerate}[\itshape i)]
\item Let's first assume that $\omega$ exists and let $\Gamma \in \Trans P$. Since
\[\omega_\Gamma \leq \epsilon \iff \Gamma \leq \Omega_\epsilon\]
for all $\epsilon \geq 0$, we see that if $\omega_\Gamma = \infty$, the set consists only of $\infty$, and if $\omega_\Gamma < \infty$, then $\Gamma \leq \Omega_{\omega_\Gamma}$ by setting $\epsilon = \omega_\Gamma$. Hence $\omega_\Gamma$ is in the set. If $\epsilon$ is in the set as well, then $\Gamma \leq \Omega_\epsilon$, and further $\omega_\Gamma \leq \epsilon$. Hence, $\omega_\Gamma$ is a lower bound, and consequently the minimum.

Now we'll assume that the minimums exist, and we'll show that the assignment
\[\omega_\Gamma = \min \big(\{\epsilon \in [0,\infty) \mid \Gamma \leq \Omega_\epsilon\} \cup \{ \infty \}\big)\]
defines a sublinear projection such that $\omega \dashv \Omega$. First, since $I \leq \Omega_0$, we get $\omega_I = 0$. Let $\Gamma, \K \in \Trans P$. If $\Gamma \leq \K$, then
\[\{\epsilon \in [0,\infty) \mid \K \leq \Omega_\epsilon\} \subseteq \{\epsilon \in [0,\infty) \mid \Gamma \leq \Omega_\epsilon\},\]
and hence $\omega_\Gamma \leq \omega_\K$.

If $\omega_\Gamma = \infty$ or $\omega_\K = \infty$, then clearly $\omega_{\Gamma\K} \leq \omega_\Gamma + \omega_\K$. If $\omega_\Gamma, \omega_\K < \infty$, then $\Gamma \leq \Omega_{\omega_\Gamma}$ and $\K \leq \Omega_{\omega_\K}$. Since $\Omega$ is a superlinear family,
\[\Omega_{\omega_\Gamma}\Omega_{\omega_\K} \leq \Omega_{\omega_\Gamma + \omega_\K},\]
and further $\Gamma\K \leq \Omega_{\omega_\Gamma + \omega_\K}$. Hence $\omega_{\Gamma\K} \leq \omega_\Gamma + \omega_\K$. This shows that $\omega$ is a sublinear projection.

Clearly
\[\{\epsilon \in [0,\infty) \mid \Gamma \leq \Omega_\epsilon\} \cup \{ \infty \} = [\omega_\Gamma,\infty],\]
and further
\[\omega_\Gamma \leq \epsilon \iff \epsilon \in [\omega_\Gamma,\infty] = \{\epsilon' \in [0,\infty) \mid \Gamma \leq \Omega_{\epsilon'}\} \iff \Gamma \leq \Omega_\epsilon\]
for all $\epsilon \geq 0$. Hence $\omega \dashv \Omega$.

\item The proof for this case is obtained easily by dualizing the proof of case \emph{i)} so we'll skip it.
\end{enumerate}
\end{proof}

\begin{prop}
Let $\Omega$ be a superlinear family and $\omega$ be a sublinear projection such that $\omega \dashv \Omega$.
\begin{enumerate}[\itshape i)]
\item Let $\Gamma, \Gamma', \K, \K' \in \Trans P$ such that $\Gamma \leq \Gamma'$ and $\K \leq \K'$. Then, for all \emph{decreasing} maps $f \colon \Dgm_{\mb P} \to \mb G$
\[\nabla_{\Gamma',\K'}f \leq \nabla_{\Gamma,\K}f.\]
\item For all \emph{decreasing} maps $f,g \colon \Dgm_{\mb P} \to \mb G$
\[\dist_E^{\Omega}(f,g) = \dist_E^{\omega}(f,g).\]
\end{enumerate}
\end{prop}
\begin{proof}
\ 
\begin{enumerate}[\itshape i)]
\item Let $(a,b) \in \Dgm_{\mb P}$. Now
\[\Gamma'^{-1}a \leq \Gamma^{-1}a \text{ and } \K b \leq \K'b,\]
so $(\Gamma^{-1}a, \K b) \leq (\Gamma'^{-1}a, \K'b)$. Since $f$ is decreasing, we get
\[\nabla_{\Gamma',\K'}f(a,b) = f(\Gamma'^{-1}a,\K'b) \leq f(\Gamma^{-1}a,\K b) = \nabla_{\Gamma,\K}f(a,b).\]
Hence $\nabla_{\Gamma',\K'}f \leq \nabla_{\Gamma,\K}f$.

\item Let $\epsilon \in [0,\infty)$. Let's first assume that $\nabla_{\Omega_\epsilon}f \leq g$ and $\nabla_{\Omega_\epsilon}g \leq f$. By setting $\Gamma = \Omega_\epsilon$ in the defining equivalence of the adjunction relation, we see that $\omega_{\Omega_\epsilon} \leq \epsilon$. Hence we can choose $\Gamma = \K = \Omega_\epsilon$ and now $\omega_\Gamma,\omega_\K \leq \epsilon$ and $\nabla_{\Gamma,\K}f \leq g$ and $\nabla_{\K,\Gamma}g \leq f$. This shows that
\[\dist_E^{\Omega}(f,g) \geq \dist_E^{\omega}(f,g).\]

Next, let's assume that there exists $\Gamma, \K \in \Trans P$ such that $\omega_\Gamma, \omega_\K \leq \epsilon$, $\nabla_{\Gamma, \K}f \leq g$ and $\nabla_{\K, \Gamma}g \leq f$. Since $\omega_\Gamma, \omega_\K \leq \epsilon$, the adjunction relation says that $\Gamma,\K \leq \Omega_\epsilon$. Hence, by part \emph{i)} of this theorem, $\nabla_{\Omega_\epsilon}f \leq \nabla_{\Gamma, \K}f$ and $\nabla_{\Omega_\epsilon}g \leq \nabla_{\K, \Gamma}g$, and further
\[\nabla_{\Omega_\epsilon}f \leq g \text{ and } \nabla_{\Omega_\epsilon}g \leq f.\]
This shows that
\[\dist_E^{\Omega}(f,g) \leq \dist_E^{\omega}(f,g).\]
\end{enumerate}
\end{proof}

\begin{cor}
For all persistence modules $F,G \colon \mb P \to \mb C$
\[\dist_E^{\Omega}(F,G) = \dist_E^{\omega}(F,G).\]
\end{cor}


\phantomsection
\addcontentsline{toc}{section}{\refname}

\bibliographystyle{halpha5}
\bibliography{mybib}{}

\end{document}